\documentclass[11pt]{article}

\usepackage{amssymb,amsmath,amsfonts,amsthm}
\usepackage{latexsym}
\usepackage{graphics}
\usepackage{indentfirst}
\usepackage{hyperref}
\usepackage{comment}
\usepackage[shortlabels]{enumitem}
\usepackage[english]{babel}
\usepackage{tikz}
\usepackage{esdiff}

\usetikzlibrary{arrows.meta}

\newtheorem*{thm*}{Theorem}
\newtheorem{thm}{Theorem}
\newtheorem{lem}[thm]{Lemma}

\newtheorem{conj}[thm]{Conjecture}

\newtheorem{ques}[thm]{Question}

\newcommand{\CC}{\mathcal{C}}

\newcommand{\N}{\mathbb{N}}

\newcommand{\R}{\mathbb{R}}

\allowdisplaybreaks

\begin{document}

\title{On the List Color Function Threshold}

\author{Hemanshu Kaul$^1$, Akash Kumar$^2$, Jeffrey A. Mudrock$^3$, Patrick Rewers$^2$, \\ Paul Shin$^2$, and Khue To$^2$}

\footnotetext[1]{Department of Applied Mathematics, Illinois Institute of Technology, Chicago, IL 60616. E-mail: {\tt {kaul@iit.edu}}}

\footnotetext[2]{Department of Mathematics, College of Lake County, Grayslake, IL 60030.}

\footnotetext[3]{Department of Mathematics and Statistics, University of South Alabama, Mobile, AL 36688.  E-mail:  {\tt {mudrock@southalabama.edu}}}

\maketitle

\begin{abstract}

The chromatic polynomial of a graph $G$, denoted $P(G,m)$, is equal to the number of proper $m$-colorings of $G$.  The list color function of graph $G$, denoted $P_{\ell}(G,m)$, is a list analogue of the chromatic polynomial that has been studied since the early 1990s, primarily through comparisons with the corresponding chromatic polynomial.  It is known that for any graph $G$ there is a $k \in \N$ such that $P_\ell(G,m) = P(G,m)$ whenever $m \geq k$.   The \emph{list color function threshold of $G$}, denoted $\tau(G)$, is the smallest $k \geq \chi(G)$ such that $P_{\ell}(G,m) = P(G,m)$ whenever $m \geq k$.  In 2009, Thomassen asked whether there is a universal constant $\alpha$ such that for any graph $G$, $\tau(G) \leq \chi_{\ell}(G) + \alpha$, where $\chi_{\ell}(G)$ is the list chromatic number of $G$.  We show that the answer to this question is no by proving that there exists a positive constant $C$ such that $\tau(K_{2,l}) - \chi_{\ell}(K_{2,l}) \ge C\sqrt{l}$ for $l \ge 16$.

\medskip



\end{abstract}

\section{Introduction}\label{intro}

In this paper all graphs are nonempty, finite, simple graphs unless otherwise noted.  Generally speaking we follow West~\cite{W01} for terminology and notation.  The set of natural numbers is $\N = \{1,2,3, \ldots \}$.  For $m \in \N$, we write $[m]$ for the set $\{1, \ldots, m \}$.  If $G$ is a graph and $S \subseteq V(G)$, we use $G[S]$ for the subgraph of $G$ induced by $S$.  If $u$ and $v$ are adjacent in $G$, $uv$ or $vu$ refers to the edge between $u$ and $v$.  We write $K_{n,l}$ for complete bipartite graphs with partite sets of size $n$ and $l$.  If $G$ and $H$ are vertex disjoint graphs, we write $G \vee H$ for the join of $G$ and $H$.

\subsection{List Coloring and Counting List Colorings} \label{basic}

In classical vertex coloring one wishes to color the vertices of a graph $G$ with up to $m$ colors from $[m]$ so that adjacent vertices in $G$ receive different colors, a so-called \emph{proper $m$-coloring}.  The \emph{chromatic number} of a graph, denoted $\chi(G)$, is the smallest $m$ such that $G$ has a proper $m$-coloring.  List coloring is a generalization of classical vertex coloring introduced independently by Vizing~\cite{V76} and Erd\H{o}s, Rubin, and Taylor~\cite{ET79} in the 1970s.  In list coloring, we associate a \emph{list assignment} $L$ with a graph $G$ so that each vertex $v \in V(G)$ is assigned a list of available colors $L(v)$ (we say $L$ is a list assignment for $G$).  We say $G$ is \emph{$L$-colorable} if there is a proper coloring $f$ of $G$ such that $f(v) \in L(v)$ for each $v \in V(G)$ (we refer to $f$ as a \emph{proper $L$-coloring} of $G$).  A list assignment $L$ is called a \emph{$k$-assignment} for $G$ if $|L(v)|=k$ for each $v \in V(G)$.  We say $G$ is \emph{$k$-choosable} if $G$ is $L$-colorable whenever $L$ is a $k$-assignment for $G$.  The \emph{list chromatic number} of a graph $G$, denoted $\chi_\ell(G)$, is the smallest $k$ such that $G$ is $k$-choosable.  It is immediately obvious that for any graph $G$, $\chi(G) \leq \chi_\ell(G)$.  Moreover, it is well-known that the gap between the chromatic number and list chromatic number of a graph can be arbitrarily large as the following result illustrates (see e.g.,~\cite{KM216} for further details).  

\begin{thm} [Folklore] \label{thm: listbipartite}
For $n \in \N$, $\chi_\ell(K_{n,l})=n+1$ if and only if $l \geq n^n$.
\end{thm}

In 1912 Birkhoff~\cite{B12} introduced the notion of the chromatic polynomial with the hope of using it to make progress on the four color problem.  For $m \in \N$, the \emph{chromatic polynomial} of a graph $G$, $P(G,m)$, is the number of proper $m$-colorings of $G$. It is well-known that $P(G,m)$ is a polynomial in $m$ of degree $|V(G)|$ (e.g., see~\cite{DKT05}). For example, $P(K_n,m) = \prod_{i=0}^{n-1} (m-i)$, $P(C_n,m) = (m-1)^n + (-1)^n (m-1)$, $P(T,m) = m(m-1)^{n-1}$ whenever $T$ is a tree on $n$ vertices, and $P(K_{2,l},m) = m(m-1)^l + m(m-1)(m-2)^l$ (see~\cite{B94} and~\cite{W01}). 

The notion of chromatic polynomial was extended to list coloring in the early 1990s by Kostochka and Sidorenko~\cite{AS90}.  If $L$ is a list assignment for $G$, we use $P(G,L)$ to denote the number of proper $L$-colorings of $G$. The \emph{list color function} $P_\ell(G,m)$ is the minimum value of $P(G,L)$ where the minimum is taken over all possible $m$-assignments $L$ for $G$.  Since an $m$-assignment could assign the same $m$ colors to every vertex in a graph, it is clear that $P_\ell(G,m) \leq P(G,m)$ for each $m \in \N$.  In general, the list color function can differ significantly from the chromatic polynomial for small values of $m$.  One reason for this is that a graph can have a list chromatic number that is much higher than its chromatic number.   On the other hand,  in 1992, Donner~\cite{D92} showed that for any graph $G$ there is a $k \in \N$ such that $P_\ell(G,m) = P(G,m)$ whenever $m \geq k$. 

It is also known that $P_{\ell}(G,m)=P(G,m)$ for all $m \in \N$ when $G$ is a cycle or chordal  (see~\cite{AS90} and~\cite{KN16}).  Moreover, if $P_{\ell}(G,m)=P(G,m)$ for all $m \in \N$, then $P_{\ell}(G \vee K_n,m)=P(G \vee K_n,m)$ for each $n, m \in \N$ (see~\cite{KM18}). 

\subsection{The List Color Function Threshold} 

We now introduce a notion that has received some attention (under different names) in the literature~\footnote{It is worth mentioning that a DP-coloring (see~\cite{DP15}) analogue of the list color function threshold was recently introduced and studied in~\cite{BH21}.}.  Given any graph $G$, the \emph{list color function number of $G$}, denoted $\nu(G)$, is the smallest $t \geq \chi(G)$ such that $P_{\ell}(G,t) = P(G,t)$.  The \emph{list color function threshold of $G$}, denoted $\tau(G)$, is the smallest $k \geq \chi(G)$ such that $P_{\ell}(G,m) = P(G,m)$ whenever $m \geq k$.  By Donner's 1992 result, we know that both $\nu(G)$ and $\tau(G)$ are well-defined for any graph $G$. Furthermore, $\chi(G) \leq \chi_\ell(G) \leq \nu(G) \leq \tau(G)$. 

In 2009, Thomassen~\cite{T09} showed that for any graph $G$, $\tau(G) \leq |V(G)|^{10} + 1$.  Then, in 2017, Wang, Qian, and Yan~\cite{WQ17} showed that for any graph $G$, $\tau(G) \leq (|E(G)|-1)/\ln(1+ \sqrt{2}) + 1$.  Two well-known open questions on the list color function can be stated using the list color function number and list color function threshold.

\begin{ques} [Kirov and Naimi~\cite{KN16}] \label{ques: threshold}
	For every graph $G$, is it the case that $\nu(G) = \tau(G)$?
\end{ques}   

\begin{ques} [Thomassen~\cite{T09}] \label{ques: universal}
	Is there a universal constant $\alpha$ such that for any graph $G$, $\tau(G) - \chi_{\ell}(G) \le \alpha$?
\end{ques}

Question~\ref{ques: threshold}, which is asking whether the list color function of a graph and the corresponding chromatic polynomial stay the same after the first point at which they are both nonzero and equal, remains open. However, the DP-coloring analogue of Question~\ref{ques: threshold} was answered in the negative in~\cite{BK21} where it was studied under the notion of chromatic adherence (see~\cite{KM19} for an introduction to the DP color function, the DP-coloring analogue of list color function). 

In \cite{T09}, it was shown that in Question~\ref{ques: universal}, $\alpha$  cannot be zero. In this paper we show that the answer to Question~\ref{ques: universal} is no in a fairly strong sense.  Specifically, we prove the following.

\begin{thm} \label{thm: K2l}
Suppose $G= K_{2,l}$ and $l\geq16$. Let $q = \lfloor l/4\rfloor$. Then, 
$$\tau(G)>\left\lfloor\left(\frac{q}{\ln(16/7)}\right)^{1/2} + 1\right\rfloor.$$
Consequently, there is a constant $C >0$ such that for each $l \geq 16$, $\tau(K_{2,l}) - \chi_{\ell}(K_{2,l}) = \tau(K_{2,l}) - 3 \geq C \sqrt{l}$. 
\end{thm}

We have made no attempt to optimize the leading constant above. However, we believe that this lower bound captures the behavior of $\tau(K_{2,l})$.

\begin{conj} \label{conj: k2l}
	$\tau(K_{2,l}) = \Theta(\sqrt{l})$ as $l \rightarrow \infty$.
\end{conj} 

In light of the bound of Wang, Qian, and Yan, Thomassen's Question~\ref{ques: universal}, and Theorem~\ref{thm: K2l}, it is natural to study the asymptotic behavior of the list color function threshold as the size of the graphs we consider tends toward infinity. 
We define the extremal functions $\delta_{max}(t) = \max\{\tau(G) - \chi_{\ell}(G) : G \text{ is a graph with at most $t$ edges}\}$ and $\tau_{max}(t) = \max\{\tau(G) : G \text{ is a graph with at most $t$ edges}\}$. By Theorem~\ref{thm: K2l} and the bound of Wang, Qian, and Yan, we know that there exist positive constants $C_1,C_2$ such that $C_1\sqrt{t} \le \delta_{max}(t) \le C_2t$ for large enough $t$. The same asymptotic bounds hold for $\tau_{max}(t)$ as well.

\begin{ques}\label{ques: delta-max}
	What is the asymptotic behavior of $\delta_{max}(t)$? 
\end{ques}

Since $\chi_{\ell}(G) = O(\sqrt{|E(G)|})$ as $|E(G)| \rightarrow \infty$, if $\tau_{max}(t) = \omega(\sqrt{t})$ as $t \rightarrow \infty$, then $\delta_{max}(t) \sim \tau_{max}(t)$ as $t \rightarrow \infty$. 

\begin{ques}\label{ques: tau-max}
What is the asymptotic behavior of $\tau_{max}(t)$? In particular, is $\tau_{max}(t) = \omega(\sqrt{t})$?
\end{ques}

Understanding $\tau(K_{n,l})$ would be the first natural candidate towards answering Questions~\ref{ques: delta-max} and~\ref{ques: tau-max}.

\section{Proof of Theorem~\ref{thm: K2l}} \label{main}

To prove a lower bound on $\tau(G)$, we need an upper bound on $P_\ell(G,m)$ that is smaller than $P(G,m)$ for some $m$. 
Our first step is to give an enumerative generalization~\footnote{While we really only need the generalization when $n=2$, we prove the result for general $n$ for completeness.} of the ``if" direction of Theorem~\ref{thm: listbipartite}. We generalize the folklore `bad' list assignment from Theorem~\ref{thm: listbipartite} and count the number of such list colorings to get an upper bound on $P_\ell(K_{n, n^nt},m)$.

\begin{lem} \label{lem: K-n-n^nt-construction-list-formula}
Let $n, m, t \in \mathbb{N}$ with $n \geq 2$ and $m \geq n + 1$, and let $G = K_{n, n^nt}$ with bipartition $X = \{x_1, \ldots, x_n\}$, $Y = \{y_1, \ldots, y_{n^nt}\}$. Let $S_k = \{m + n(k - 2) + \ell : \ell \in [n]\}$ for each $k \in [n]$, and let $A = \{\{s_1, \ldots, s_n\} : s_k \in S_k \text{ for each } k \in [n]\}$. Suppose $A = \{A_0, \ldots, A_{n^n - 1}\}$. Let $L$ be the $m$-assignment for $G$ defined by $L(x_k) = [m - n] \cup S_k$ for each $k \in [n]$ and $L(y_k) = [m - n] \cup A_{\lfloor (k - 1) / t \rfloor}$ for each $k \in [n^n t]$. Then \footnotemark
\begin{align*}
    P(G, L)
    &= n^n \prod_{i = 0}^n (m - i)^{t\binom{n}{i}(n - 1)^{n - i}} \\
    & + \sum_{N = 1}^n \sum_{S = 0}^{n - N} \left[ n^S \binom{n}{S} \binom{m - n}{N} \left( \sum_{i = 0}^{N - 1} (-1)^i \binom{N}{i} (N - i)^{n - S} \right) \right. \\
    & \cdot \left. \prod_{i = 0}^{S} (m - N - i)^{t \binom{S}{i} (n - 1)^{S - i} n^{n - S}} \right].
\end{align*}
\footnotetext{If $a < b$, then we interpret $\binom{a}{b}$ as being equal to zero.}
\end{lem}

\begin{proof}
Let $\mathcal{C}$ be the set of all proper $L$-colorings of $G$. Let $T = \{(0, n)\} \cup \{(N, S) \in \mathbb{Z} \times \mathbb{Z} : 1 \leq N \leq n \text{ and } 0 \leq S \leq n - N\}$. For each $(N, S) \in T$, let $\mathcal{T}_{(N, S)}$ be the set of proper $L$-colorings $f$ of $G$ such that $\lvert f(X) \cap [m - n] \rvert = N$ and $\lvert f(X) \cap \bigcup_{k = 1}^n S_k \rvert = S$. Notice that $P(G, L) = \lvert \mathcal{C} \rvert = \sum_{(N, S) \in T} \lvert \mathcal{T}_{(N, S)} \rvert$.

Let $L'$ be the restriction of $L$ to $X$. We compute $\lvert \mathcal{T}_{(0, n)} \rvert$ in two steps: first, we count the number of proper $L'$-colorings $h$ of $G[X]$ such that $h(X) \subseteq \bigcup_{k = 1}^n S_k$. Then, given a proper $L'$-coloring $h$ of $G[X]$ such that $h(X) \subseteq \bigcup_{k = 1}^n S_k$, we count the number of proper $L$-colorings $f$ of $G$ such that $f(v) = h(v)$ for each $v \in X$. We will find that the number obtained in the second step does not depend on $h$, so that $\lvert \mathcal{T}_{(0, n)} \rvert$ equals the number obtained in the first step times the number obtained in the second step.

For step one, notice that $\lvert S_k \rvert = n$ for each $k \in [n]$. So, the number of proper $L'$-colorings $h$ of $G[X]$ such that $h(X) \subseteq \bigcup_{k = 1}^n S_k$ is $\prod_{k = 1}^n \lvert S_k \rvert = n^n$. For step two, suppose $h$ is a proper $L'$-coloring of $G[X]$ such that $h(X) \subseteq \bigcup_{k = 1}^n S_k$. For each $k \in [n]$, suppose $h(x_k) = s_k$. For each integer $i$ with $0 \leq i \leq n$, let $R_i = \{j \in [n^nt] : \lvert L(y_j) \cap \{s_k : k \in [n]\} \rvert = i\}$. Notice that $R_0, \ldots, R_n$ form a partition of $[n^nt]$, and that $\lvert R_i \rvert = t \cdot \lvert \{j \in [n^n - 1] \cup \{0\} : \lvert A_j \cap \{s_k : k \in [n]\} \rvert = i\} \rvert$ for each integer $i$ with $0 \leq i \leq n$. By the definition of $A$, it is easy to see $\lvert R_i \rvert = t \binom{n}{i} (n - 1)^{n - i}.$  Then, the number of proper $L$-colorings $f$ of $G$ such that $f(v) = h(v)$ for each $v \in X$ is given by
\begin{align*}
    \prod_{i = 0}^n (m - i)^{\lvert R_i \rvert} = \prod_{i = 0}^n (m - i)^{t \binom{n}{i} (n - 1)^{n - i}}
\end{align*}
from which we conclude
\begin{align*}
    \lvert \mathcal{T}_{(0, n)} \rvert = n^n \prod_{i = 0}^n (m - i)^{t \binom{n}{i} (n - 1)^{n - i}}.
\end{align*}

We now compute $\lvert \mathcal{T}_{(N, S)} \rvert$ for arbitrary $(N, S) \in T \setminus \{(0, n)\}$. To do so, we again employ a two-step process: first, we count the number of proper $L'$-colorings $h$ of $G[X]$ such that $\lvert h(X) \cap [m - n] \rvert = N$ and $\lvert h(X) \cap \bigcup_{k = 1}^n S_k \rvert = S$. Then, given a proper $L'$-coloring $h$ of $G[X]$ such that $\lvert h(X) \cap [m - n] \rvert = N$ and $\lvert h(X) \cap \bigcup_{k = 1}^n S_k \rvert = S$, we count the number of proper $L$-colorings $f$ of $G$ such that $f(v) = h(v)$ for each $v \in X$. Again, we will find that the number obtained in the second step does not depend on $h$, so that $\lvert \mathcal{T}_{(N, S)} \rvert$ equals the number obtained in the first step times the number obtained in the second step.

For step one, we can generate all such proper $L'$-colorings $h$ of $G[X]$ via the following four-part process: first, choose a subset $P_s$ of $[n]$ of size $S$, and let $P_o = [n] - P_s$. Secondly, choose a subset $O$ of $[m - n]$ of size $N$. Thirdly, color the vertices in $\{x_k : k \in P_o\}$ with the colors in $O$ such that each color in $O$ is used at least once. Lastly, for each $k \in P_s$, color $x_k$ with a color in $S_k$. The first part can be done in $\binom{n}{S}$ ways. The second part can be done in $\binom{m - n}{N}$ ways. By some simple counting and the Inclusion-Exclusion Principle, the third part can be done in
\begin{align*}
    \sum_{i = 0}^N (-1)^i \binom{N}{i} (N - i)^{n - S} = \sum_{i = 0}^{N - 1} (-1)^i \binom{N}{i} (N - i)^{n - S}
\end{align*}
ways. Finally, for each $k \in P_s$, there are $\lvert S_k \rvert = n$ ways to color $x_k$ with a color in $S_k$. Thus, the final part can be done in $n^S$ ways. Hence, the number of proper $L'$-colorings $h$ of $G[X]$ such that $\lvert h(X) \cap [m - n] \rvert = N$ and $\lvert h(X) \cap \bigcup_{k = 1}^n S_k \rvert = S$ is
\begin{align*}
    \binom{n}{S} \binom{m - n}{N} \left( \sum_{i = 0}^{N - 1} (-1)^i \binom{N}{i} (N - i)^{n - S} \right) n^S.
\end{align*}

For step two, suppose $h$ is a proper $L'$-coloring of $G[X]$ such that $\lvert h(X) \cap [m - n] \rvert = N$ and $\lvert h(X) \cap \bigcup_{k = 1}^n S_k \rvert = S$. Let $P_o = \{k \in [n] : h(x_k) \in [m - n]\}$ and $P_s = \{k \in [n] : h(x_k) \in S_k\}$. Notice that $P_o$ and $P_s$ form a partition of $[n]$. Suppose $h(X) \cap [m - n] = \{o_1, \ldots, o_N\}$. For each $k \in P_s$, suppose $h(x_k) = s_k$. For each $i \in \mathbb{Z}$ with $0 \leq i \leq \lvert P_s \rvert = S$, let $R_i = \{j \in [n^nt] : \lvert L(y_j) \cap \{s_k : k \in P_s\} \rvert = i\}$. Notice that $R_0, \ldots, R_{S}$ form a partition of $[n^nt]$, and that $\lvert R_i \rvert = t \cdot \lvert \{j \in [n^n - 1] \cup \{0\} : \lvert A_j \cap \{s_k : k \in P_s\} \rvert = i\} \rvert$ for each $i \in \mathbb{Z}$ with $0 \leq i \leq S$. By the definition of $A$, it is easy to see $\lvert R_i \rvert = t \binom{S}{i} (n - 1)^{S - i} n^{n - S}.$  Then, the number of proper $L$-colorings $f$ of $G$ such that $f(v) = h(v)$ for each $v \in X$ is given by
\begin{align*}
\prod_{i = 0}^S (m - N - i)^{\lvert R_i \rvert} = \prod_{i = 0}^S (m - N - i)^{t \binom{S}{i} (n - 1)^{S - i} n^{n - S}},
\end{align*}
from which we conclude
\begin{align*}
    \lvert \mathcal{T}_{(N, S)} \rvert = n^S \binom{n}{S} \binom{m - n}{N} \left( \sum_{i = 0}^{N - 1} (-1)^i \binom{N}{i} (N - i)^{n - S} \right) \prod_{i = 0}^S (m - N - i)^{t \binom{S}{i} (n - 1)^{S - i} n^{n - S}}.
\end{align*}
The result follows.
\end{proof}

We can use Lemma~\ref{lem: K-n-n^nt-construction-list-formula} to find appropriate $m,$ $n,$ and $t$ such that $P(G, L) < P(G,m)$, which would imply $\tau(G) > m$.  Since the focus of Theorem~\ref{thm: K2l} is $n=2$, we will now slightly generalize the list assignment constructed in the statement of Lemma~\ref{lem: K-n-n^nt-construction-list-formula} in the case $n=2$. The notion of `balanced' list assignment given below captures the essence of what makes this list assignment `bad' as well as nice to work with.

Suppose $G = K_{2,l}$, the bipartition of $G$ is $\{x_1,x_2\}, \{y_1, \ldots, y_l\}$, and $L$ is an $m$-assignment for $G$ such that $L(x_1) = [m]$ and $L(x_2) = [m-2] \cup \{m+1,m+2\}$. Let $z_1 = |\{j\in[l] : L(y_j) = [m-2] \cup \{m-1,m+1\}\}|$, $z_2 = |\{j\in[l] : L(y_j) = [m-2] \cup \{m-1,m+2\}\}|$, $z_3 = |\{j\in[l] : L(y_j) = [m-2] \cup \{m,m+1\}\}|$, and $z_4 = |\{j\in[l] : L(y_j) = [m-2] \cup \{m,m+2\}\}|$. Then, we say the list assignment $L$ is \emph{balanced} if $\sum_{i=1}^{4} z_i = l$ and $|z_j-z_i| \leq 1$ whenever $i,j\in[4]$.    

We will now use the formula and list assignment for $G = K_{2,4t}$ in Lemma~\ref{lem: K-n-n^nt-construction-list-formula}, to determine how large $t$ must be to ensure the existence of a balanced $m$-assignment $L$ for $G$ that demonstrates $P_\ell(G,m) < P(G,m)$. 


\begin{lem}\label{lem;4t}
	Suppose $G = K_{2,4t}$ and $m\geq3$. If $$t > \max\left\{\frac{\ln(\epsilon/(4(m-2)))}{2\ln((m-2)/(m-1))},\frac{\ln((2-\epsilon)/4)}{\ln(1-1/(m-1)^2)}\right\}$$ for some real number $\epsilon$ with $0<\epsilon<2$, then there is a balanced $m$-assignment $L$ for $G$ such that $P(G,L)<P(G,m)$.
\end{lem}
\begin{proof}
	Suppose $G = K_{2, 4t}$ and the bipartition of $G$ is $\{x_1, x_2\}$, $\{y_1, \ldots, y_{4t}\}$.  Clearly, $P(G, m) = m(m - 1)^{4t} + m(m - 1)(m - 2)^{4t}$.  If $L$ is the $m$-assignment for $G$ described in the statement of Lemma~\ref{lem: K-n-n^nt-construction-list-formula}, then $L$ is a balanced $m$-assignment and
	\begin{align*}
		P(G, L)
		&= (m - 2)(m - 1)^{4t} + (m - 3)(m - 2)^{4t + 1} \\
		& + 4(m - 2)^{2t + 1}(m - 1)^{2t} + 4(m - 2)^t(m - 1)^{2t}m^t
	\end{align*} by Lemma~\ref{lem: K-n-n^nt-construction-list-formula}. We will show $P(G,L) < P(G,m)$.  Let $b=1+1/(m-1)$ and $s=1-1/(m-1).$  Notice 
	$$t > \frac{\ln(\epsilon/(4(m-2)))}{2\ln((m-2)/(m-1))}\;\;\;\;\text{which implies}\;\;\;\;4(m-2)s^{2t}<\epsilon,\; \text{as }\;\; 2 \ln((m -
	2)/(m - 1)) < 0.$$
	Also
	$$t > \frac{\ln((2-\epsilon)/4)}{\ln(1-1/(m-1)^2)}\;\;\;\;\text{which implies}\;\;\;\;4s^tb^t<2-\epsilon.$$
	Combining these inequalities yields
	$$4(m-2)s^{2t}+4s^tb^t<2\;\;\;\;\text{which implies}\;\;\;\;4(m-2)s^{2t}+4s^tb^t<2+(4m-6)s^{4t}.$$
	Then,
	\begin{align*}
		&4(m-1)^{2t}(m-2)^{2t+1}+4m^t(m-2)^t(m-1)^{2t} \\
		&<2(m-1)^{4t}+m(m-1)(m-2)^{4t}-(m-2)^{4t+1}(m-3).
	\end{align*}
	From which we obtain 
	\begin{align*}
		&(m-2)(m-1)^{4t}+(m-2)^{4t+1}(m-3)+4(m-1)^{2t}(m-2)^{2t+1}+4m^t(m-2)^t(m-1)^{2t}\\
		&<m(m-1)^{4t}+m(m-1)(m-2)^{4t}
	\end{align*}
	as desired.
\end{proof}

We now establish some notation that will be used for the remainder of the paper.  Suppose $G=K_{2,l}$, the bipartition of $G$ is $\{x_1,x_2\}, \{y_1, \ldots, y_l\}$, and $L$ is an $m$-assignment for $G$. For each $(a_1,a_2)\in L(x_1) \times L(x_2)$, let $\mathcal{C}_{(a_1,a_2)}$ be the set of proper
$L$-colorings of $G$ in which $x_i$ is colored with $a_i$ for each $i \in [2]$. Notice $P(G, L) = \sum_{(a_1,a_2)\in L(x_1) \times L(x_2) }|\CC_{(a_1, a_2)}|.$

Generally speaking, our strategy for proving Theorem~\ref{thm: K2l} is inductive.  We wish to show that if there is a balanced $m$-assignment $L$ for $G = K_{2,4t}$ that demonstrates $P_\ell(G,m) < P(G,m)$ (Lemma~\ref{lem: K-n-n^nt-construction-list-formula} will be the key to proving such an $L$ exists) and $t$ is sufficiently large, then for any $l \geq 4t$ there is a balanced $m$-assignment $L'$ for $G' = K_{2,l}$ that demonstrates $P_\ell(G',m) < P(G',m)$.  The next two lemmas make the inductive idea precise.

\begin{lem}\label{lem; balanced}
Suppose $G = K_{2,l}$ and $G' = K_{2,l+1}$.  If $L$ is a balanced $m$-assignment for $G$ with $m \geq 3$, $P(G,L) < P(G,m)$, $\epsilon \in (0,2)$ , and $l$ satisfies 
$$\left\lfloor \frac{l}{4} \right \rfloor > \max\left\{\frac{\ln(\epsilon/(2(m-2)))}{2\ln((m-2)/(m-1))},\frac{\ln((2-\epsilon)/4)}{\ln(1-1/(m-1)^2)}\right\},$$ 
then there is a balanced $m$-assignment $L'$ for $G'$ such that $P(G',L')< P(G',m)$.
\end{lem}

\begin{proof}
For simplicity, suppose the bipartitions of $G$ and $G'$ are $\{x_1,x_2\}, \{y_1, \ldots, y_l\}$ and $\{x_1,x_2\}, \{y_1, \ldots, y_{l+1}\}$ respectively. We know that $P(G,L) < P(G,m) = m(m-1)^l+m(m-1)(m-2)^l$. We also know that $P(G',m) = m(m-1)^{l+1}+m(m-1)(m-2)^{l+1}.$ As such $$P(G',m)-P   (G,m)=m(m-1)^l(m-2)+m(m-1)(m-2)^l(m-3).$$ 
Let $c_{(i,j)} = |\mathcal{C}_{(i,j)}|$, with regard to $G$. We know $P(G,L) = \sum_{(i,j)\in L(x_1)\times L(x_2)} c_{(i,j)}.$ Without loss of generality assume $z_1 \leq z_j$ for each $j\in \{2,3,4\}$. Let $L'$ be the $m$-assignment for $G'$ given by $L'(v) = L(v)$ if $v\in V(G)$ and $L'(y_{l+1}) = [m-2]\cup \{m-1, m+1\}$. Clearly $L'$ is a balanced $m$-assignment for $G'$.

With some simple counting, we see that:
\begin{align*} 
P(G',L') &= (m-1)\sum_{i=1}^{m-2}c_{(i,i)} + (m-2)\sum_{\substack{(i,j)\in [m-2]^2\\
i\neq j}}c_{(i,j)}+ (m-2)\sum_{i=1}^{m-2}c_{(m-1,i)}\\
 &+ (m-1)\sum_{i=1}^{m-2}c_{(m,i)} + (m-2)\sum_{i=1}^{m-2}c_{(i,m+1)} + (m-1)\sum_{i=1}^{m-2}c_{(i,m+2)}\\
&+ (m-2)c_{(m-1,m+1)} + (m-1)[c_{(m-1,m+2)} + c_{(m,m+1)}] + mc_{(m,m+2)}\\ 
&= (m-2)P(G,L)+ 2c_{(m,m+2)}+c_{(m-1,m+2)}+c_{(m,m+1)}+\sum_{i=1}^{m-2}c_{(i,m+2)}\\
&+\sum_{i=1}^{m-2}c_{(m,i)}+\sum_{i=1}^{m-2}c_{(i,i)}.
\end{align*}
This implies $P(G',L') - P(G,L) = (m-3)P(G,L) + J \leq (m-3)[m(m-1)^l+m(m-1)(m-2)^l] + J$, where $J = P(G',L')-(m-2)P(G,L).$  We will show if 
$$\left\lfloor \frac{l}{4} \right \rfloor > \max\left\{\frac{\ln(\epsilon/(2(m-2)))}{2\ln((m-2)/(m-1))},\frac{\ln((2-\epsilon)/4)}{\ln(1-1/(m-1)^2)}\right\},$$
then $J<m(m-1)^l$. Notice, if $J<m(m-1)^l$ we have $P(G',L') - P(G,L) < P(G',m)-P(G,m)$ which implies $P(G',L')<P(G',m).$  Since $z_1 \leq z_j$ where $j \in [4]$, $z_1 = \lfloor l/4 \rfloor$. So,
$$z_1 > \frac{\ln(\epsilon/(2(m-2)))}{2\ln((m-2)/(m-1))}\;\;\;\;\text{which implies}\;\;\;\;
2(m-2)\left(\frac{m-2}{m-1}\right)^{2z_1} < \epsilon.$$
Since $2z_1 \leq z_2 + z_4$, $2z_1 \leq z_3 + z_4$, and $(m-2)/(m-1) < 1,$ we have
\begin{equation} \label{eqn:two}
    (m-2)\left(\frac{m-2}{m-1}\right)^{z_2+z_4}
+(m-2)\left(\frac{m-2}{m-1}\right)^{z_3+z_4} < \epsilon.
\end{equation}
Similarly,
$$
z_{1} > \frac{\ln((2-\epsilon)/4)}{\ln(1-1/(m-1)^2)}\;\;\;\;\text{which implies}\;\;\;\; 4\left(\frac{m(m-2)}{(m-1)^2}\right)^{z_1} < 2-\epsilon.
$$
Now, let $b=1+1/(m-1)$ and $s=1-1/(m-1)$. The most recent inequality becomes $4(bs)^{z_1}<2-\epsilon$, and since $s<1$, 
$2(bs)^{z_1} \geq 2b^{z_1}s^{z_4}$. We will show that $b^{z_3}s^{z_2}+b^{z_2}s^{z_3} \leq 2(bs)^{z_1}$. Notice that $z_2=z_3$ or $\max\{z_2,z_3\} = z_1+1$ and $\min\{z_2,z_3\} = z_1$.

Assume $z_2=z_3$. Since $bs < 1$, $b^{z_3}s^{z_2}+b^{z_2}s^{z_3} =2(bs)^{z_2} \leq 2(bs)^{z_1}$.
Now, without loss of generality, assume $z_2=z_1$ and $z_3=z_1+1$. Then, $b^{z_3}s^{z_2}+b^{z_2}s^{z_3} = (bs)^{z_1}(b+s) = 2(bs)^{z_1}.$ So, we have $b^{z_3}s^{z_2}+b^{z_2}s^{z_3} \leq 2(bs)^{z_1}$.
As a result, $2b^{z_1}s^{z_4} + b^{z_3}s^{z_2}+b^{z_2}s^{z_3} < 2-\epsilon$. This along with (\ref{eqn:two})
implies
$$2b^{z_1}s^{z_4} + b^{z_3}s^{z_2}+b^{z_2}s^{z_3} + (m-2)s^{z_2+z_4} + (m-2)s^{z_3+z_4} < 2.$$
This implies
\begin{align*}
&2m^{z_1}(m-1)^{z_2+z_3}(m-2)^{z_4}+m^{z_3}(m-1)^{z_1+z_4}(m-2)^{z_2}+m^{z_2}(m-1)^{z_1+z_4}(m-2)^{z_3}\\
&+ (m-2)(m-1)^{z_1+z_3}(m-2)^{z_2+z_4}+(m-2)(m-1)^{z_1+z_2}(m-2)^{z_3+z_4}<2(m-1)^l
\end{align*}
which implies
\begin{align*}
&2m^{z_1}(m-1)^{z_2+z_3}(m-2)^{z_4}+m^{z_3}(m-1)^{z_1+z_4}(m-2)^{z_2}+m^{z_2}(m-1)^{z_1+z_4}(m-2)^{z_3}\\
&+ (m-2)(m-1)^{z_1+z_3}(m-2)^{z_2+z_4}+(m-2)(m-1)^{z_1+z_2}(m-2)^{z_3+z_4}+(m-2)(m-1)^{l}\\
& < m(m-1)^l.
\end{align*}
Recall $J=2c_{(m,m+2)}+c_{(m-1,m+2)}+c_{(m,m+1)}+\sum_{i=1}^{m-2}c_{(i,m+2)} + \sum_{i=1}^{m-2}c_{(m,i)}+\sum_{i=1}^{m-2}c_{(i,i)}.$
Thus, $J<m(m-1)^l$.
\end{proof} 

\begin{lem}\label{lem;threshsold}
Suppose $G = K_{2,r}$ and there is a balanced $m$-assignment $L$ for $G$ with $m\geq3$ such that $P(G,L) < P(G,m)$. Suppose there is an $ \epsilon \in (0,2)$ such that $r$ satisfies 
$$\left\lfloor \frac{r}{4} \right \rfloor > \max\left\{\frac{\ln(\epsilon/(2(m-2)))}{2\ln((m-2)/(m-1))},\frac{\ln((2-\epsilon)/4)}{\ln(1-1/(m-1)^2)}\right\}.$$ 
Then, for each $l \geq r$, if $G' = K_{2,l}$, there is a balanced $m$-assignment $L'$ for $G'$ such that $P(G',L')< P(G',m)$. Consequently, $P_{\ell}(K_{2,l},m) < P(K_{2,l},m)$ (i.e., $\tau (K_{2,l}) > m$) whenever $l \geq r$.
\end{lem}
\begin{proof}
The proof is by induction on $l$. When $l=r$ the desired statement is true since $G' = K_{2,r}$. Suppose $l>r$ and the desired statement holds for all natural numbers greater than  $r-1$ and less than $l$. Since $l-1\geq r$, there is a balanced $m$-assignment, $L$, for $H=K_{2,l-1}$ such that $P(H,L)<P(H,m)$. Since there is an $\epsilon \in (0,2)$ such that 
$$\left\lfloor\frac{l-1}{4}\right\rfloor\geq \left\lfloor \frac{r}{4} \right \rfloor > \max\left\{\frac{\ln(\epsilon/(2(m-2)))}{2\ln((m-2)/(m-1))},\frac{\ln((2-\epsilon)/4)}{\ln(1-1/(m-1)^2)}\right\},$$ Lemma~\ref{lem; balanced} implies there is a balanced $m$-assignment $L'$ for $G'= K_{2,l}$ such that $P(G',L')<P(G',m).$
\end{proof}

The next lemma follows immediately from Lemmas~~\ref{lem;threshsold} and~\ref{lem;4t}.

\begin{lem} \label{lem; paul}
If $t \in \mathbb{N}$ , $m\geq 3$, and $$t > \max\left\{\frac{\ln(\epsilon/(4(m-2)))}{2\ln((m-2)/(m-1))},\frac{\ln((2-\epsilon)/4)}{\ln(1-1/(m-1)^2)}\right\}$$ for some $\epsilon \in (0,2)$, then $\tau(K_{2,l}) > m$ whenever $l \geq 4t$.
\end{lem}

Now, we are ready to prove Theorem~\ref{thm: K2l}.

\begin{proof}
With the intent of using Lemma~\ref{lem; paul}, we will show $$q>\max\left\{\frac{\ln(\epsilon/(4(m-2)))}{2\ln((m-2)/(m-1))},\frac{\ln((2-\epsilon)/4)}{\ln(1-1/(m-1)^2)}\right\}$$
when $\epsilon = 1/4$, and $m=\left\lfloor\left(q/\ln(16/7)\right)^{1/2} + 1\right\rfloor.$ Notice $l\geq 16$ implies $q \geq 4$ and $m\geq3$. Clearly, $q \geq (m-1)^2\ln(16/7).$

Let $f: (2,\infty) \rightarrow \R$ and $g: (2, \infty) \rightarrow \R$ be given by $f(x) = (x-1)\ln(16/7)$ and $g(x) = (1/2)\ln(16(x-2)).$  Suppose $h: (2, \infty) \rightarrow \R$ is given by $h(x)= f(x) - g(x)$, we have $h'(x)=\ln(16/7) - 1/(2(x-2)).$ Notice $h'(x) > 0$ when $x \geq 3$, and $h(3)=2\ln(16/7)-\ln(16)/2 > 0$. So, $f(x) > g(x)$ when $x \geq 3$.

Now, since $q \geq(m-1)f(m)$, $q \geq (m-1)^2\ln(16/7)>((m-1)/2)\ln(16(m-2))$.
Let $\epsilon = 1/4$. Using the fact that $\ln(1+x) < x$ when $x \neq 0$ and $x>-1$, we have
$$(m-1)^2\ln(16/7) = \frac{\ln(7/16)}{-1/(m-1)^2} > \frac{\ln((2-\epsilon)/4)}{\ln(1-1/(m-1)^2)}$$
and
$$(m-1)\frac{\ln(16(m-2))}{2} = \frac{\ln(1/(16(m-2)))}{-2/(m-1)}>  \frac{\ln(\epsilon/(4(m-2)))}{2\ln((m-2)/(m-1))}.$$
Since $q=\lfloor l/4\rfloor$, $l\geq 4q$. So Lemma~\ref{lem; paul} implies
$$\tau(K_{2,l})>\left\lfloor\left(\frac{q}{\ln(16/7)}\right)^{1/2} + 1\right\rfloor.$$
\end{proof}

{\bf Acknowledgment.}  This paper is based on a research project conducted with undergraduate students Akash Kumar, Patrick Rewers, Paul Shin, and Khue To at the College of Lake County during the summer and fall of 2021. The support of the College of Lake County is gratefully acknowledged. The authors also thank Dan Cranston and Seth Thomason for helpful conversations.

\end{document}